\documentclass[review]{elsarticle}

\usepackage{lineno,hyperref}
\usepackage{amsthm}
\usepackage{amssymb}
\usepackage{mathtools}
\usepackage{color}

\newtheorem{Remark}{Remark}

\newtheorem{Theorem}{Theorem}
\newtheorem{Example}{Example}
\newtheorem{Corollary}[Theorem]{Corollary}
\newtheorem{Lemma}[Theorem]{Lemma}

\modulolinenumbers[5]

\journal{arXiv}

\begin{document}

\begin{frontmatter}

\title{Global Controllability for General Nonlinear Systems}
%%% Group authors per affiliation:

\author[mymainaddress]{Yuanyuan Liu}
\cortext[mycorrespondingauthor]{Corresponding author}
\ead{liuyuanyuan@mail.nwpu.edu.cn}

\author[mymainaddress]{Wei Zhang}
\ead{weizhangxian@nwpu.edu.cn}

\address[mymainaddress]{Northwestern Polytechnical University, Xi’an 710072, China.}

\begin{abstract}
This note studies the global controllability of a general nonlinear system by extending it to affine one. The state space of the obtained affine system admits a nature foliation, each leaf of which  is diffeomorphic to the state space of the original system. Through this foliation, the global controllability of these two kinds of systems are closely related and we prove that they are indeed equivalent. The result is then extended to the case with bounded inputs to make it practically more useful. To demonstrate the power of our approach, several examples are presented.
\end{abstract}


\end{frontmatter}

%\linenumbers

\section{Introduction}\label{Sec1}
Controllability is one of the central issue in modern control theory. For linear systems, it has been fully studied and it turns out that the controllability of a linear system is completely characterized by the rank of its controllability matrix. As for nonlinear systems, things are very different. First, the reachability of linear systems is symmetric (which means that $x_0$ is reachable by $x_1$ if and only if $x_1$ is reachable by $x_0$) while in general it is not the case for nonlinear systems. Second, there is a difference between local and global controllability for nonlinear systems. 

For global controllability, existing researches mainly focus on a special class of nonlinear systems known as affine systems, that is, control systems of the following form
\begin{equation}\label{Eq1}
\dot{x}=f(x)+g(x)u,
\end{equation}
where $x$ is the state and $u$ is the control. (Rigorous definitions will be given in Section \ref{Sec2}.) Although, as far as we know, a complete characterization of global controllability for system \eqref{Eq1} is still lacking, many insightful results have been achieved. For example,  the controllability problem of switched linear systems has been completely solved in \cite{Sun2002Controllability} and the criterion is very similar to that of linear ones.  For codimension-$1$ affine systems, i.e., that with $n$ state variables and $n-1$ independent inputs, necessary and sufficient conditions are given (separately) in \cite{Hunt1982n-Dimensional}. However, for the $2$-dimensional case, a simple and easily verifiable condition which is both necessary and sufficient is obtained in \cite{Sun2007Necessary}. The author proved an assertion which is essential for the sufficiency of the condition, that is, any control curve separates the plane into two disjoint components. The result was then extended to $3$-dimensional case under some mild assumptions (see \cite{Sun2016On}). To our best knowledge, the closest result to a complete solution to this problem was achieved in \cite{Cheng2006Global} which, to some extent, covers all the aforementioned results.

In the study of controllability of nonlinear systems, Chow's theorem (see \cite{Chow1940Uber}) plays a key role.  It gives a sufficient condition for distributions to be globally controllable, but the condition needs not to be necessary. Indeed, counterexamples are given in \cite{Sussmann1973Orbits} along with a condition which is both sufficient and necessary. The ideal is to extend the given distribution to a new and ``bigger'' one. This ideal is essential. However, the controllability of system \eqref{Eq1} is different from that of a distribution. Since the reachability of the latter is symmetric while generally it is not the case for the former. To overcome this difficulty, the concept of foliation was introduced (see \cite{Emel'yanov1988Analysis,Emel'yanov1992Application}). It decomposes the controllability of a nonlinear system into two directions, that is, directions along and traverse the leaves of a foliation. This concept turns out to be powerful in analysis of nonlinear systems, see \cite{Rampazzo2012Control,Chaib2007Invertibility,Chaib2009Failure,Sundarapandian2002Global} for some applications of it. For a comprehensive discussion in the global controllability of nonlinear systems, we refer to the books \cite{Cheng2010Analysis,Coron2007Control,Sontag1998Mathematical}.

This note concerns the global controllability of a general nonlinear system, i.e. control system of the following form,
\begin{equation}\label{Eq2}
\dot{x}=f(x,u),
\end{equation}
where $x$ is the state and $u$ is the control. (Rigorous definitions will be given in Section \ref{Sec2}.) We are going to prove the equivalence between the global controllability of system \eqref{Eq2} and a specific affine system using some of the aforementioned techniques. 

The organization of the note is as follows. In Section \ref{Sec2}, we review some basic preliminaries and fix some notations for statement ease. Our main result along with some illustrative examples are presented in Section \ref{Sec3}. The proof of the main result is given in Section \ref{Sec4}. Finally, conclusions are given in Section \ref{Sec5}. 

\section{Preliminaries and Notations}\label{Sec2} 

By a general nonlinear control system, we mean system of the form \eqref{Eq2}, where the state $x$ belongs to a paracompact manifold $M$, called the state 
space of the system. Throughout the note the word ``manifold '' means smooth, finite-dimensional, second countable, connected and Hausdorff. We use $n$ to denote the dimension of $M$. The control $u$ is a function valued in $\mathcal{C}$, called the control space and, each component $u_i$ of $u$ belongs to $\mathcal{U}$, called the set of feasible controls. In the note, $\mathcal{C}$ is always assumed to be the Euclidean space $\mathbb{R}^m$ (except in Corollary \ref{Coro2} and its proof) and $\mathcal{U}$ the set of piecewise constant functions. (In the literature, $\mathcal{U}$ is usually assumed to be the set of measurable or piecewise continuous functions, etc. But the choice of $\mathcal{U}$ doesn't affect the controllability much, see \cite{Sussmann1979Subanalytic,Sussmann1987Reachability}.) Moreover, $f:M\times\mathbb{R}^m\rightarrow TM$ is a smooth mapping which assigns, to each $x\in M$ and $u\in\mathbb{R}^m$ a tangent vector $f(x,u)$ in $T_xM$. Here, $TM$ denotes the tangent bundle of $M$ and $T_xM$ the tangent space of $M$ at $x$.

Affine control system is a special nonlinear system which is of the form \eqref{Eq1}, where $x$ and $u$ are as previously. In addition, $f(\cdot)$ and $g_1(\cdot),\cdots,g_m(\cdot)$ are smooth vector fields on $M$. Here, $g_i,i=1,\cdots,m$ denotes the $i$-th component of $g$.

An absolutely continuous curve $x:[0,T]\rightarrow M$ is said to be a trajectory of system \eqref{Eq2}, if there exists a feasible control $u:[0,T]\rightarrow\mathbb{R}^m$ such that $\dot{x}(t)=f(x(t),u(t))$ for almost every $t\in[0,T]$. In this case, we call it the trajectory of system \eqref{Eq2} associated with control $u$. Given two points $x_0,x_1\in M$, we say that $x_1$ is reachable by $x_0$, if there exists a feasible control $u:[0,T]\rightarrow\mathbb{R}^m$ such that the associated trajectory $x:[0,T]\rightarrow M$ satisfies $x(0)=x_0$ and 
$x(T)=x_1$. The set of all points in $M$ that are reachable by $x_0$ is called the reachable set of $x_0$ and is denoted by $R(x_0)$. The system \eqref{Eq2} is said to be globally controllable, if for each $x\in M$, $R(x)=M$. In the rest of the note, the word ``controllability'' always means ``global controllability'' and ``controllable'' always means ``globally controllable'', unless otherwise specified.

In the note,  we shall use $\mathbb{R}_+$ to denote the subset of  $\mathbb{R}$ that contains all the positive real numbers and $\mathbb{R}^l_+$ to denote its $l$-fold Cartesian product $\underbrace{\mathbb{R}_+\times\cdots\times\mathbb{R}_+}_{l}$. We shall use $I_\epsilon^l$ to denote the open subset $\underbrace{(-\epsilon,\epsilon)\times\cdots\times(-\epsilon,\epsilon)}_{l}$ of $\mathbb{R}^l$. We simply use $I^l$ when the specific value of $\epsilon>0$ is of no great importance.

Let $F$ be a family of vector fields on $M$, then $X$ belongs to $F^l$, denoted by $X\in F^l$, if there exist $l$ elements $X_1,\cdots,X_l\in F$ such that $X=(X_1,\cdots,X_l)$. Let $t=(t_1,\cdots,t_l)\in\mathbb{R}^l$, we use $e^{tX}$ to denote $e^{t_1X_1}\circ\cdots\circ e^{t_lX_l}$ and $e^{tX}_\ast$ to denote its tangent map. A mapping $e^{tX}:M\times\mathbb{R}\rightarrow M$ is called physically realizable by system \eqref{Eq2} if $t\ge0$ and 
$$X\in\left\{f(x,u)|u\in\mathcal{U}^m\right\}.$$

\section{Main Results and Examples}\label{Sec3}

To formulate our main result, we extend system \eqref{Eq2} to the following affine system by adding an integrator,
\begin{equation}\label{Eq3}
\begin{split}
\dot{x}&=f(x,y)\\
\dot{y}&=v.
\end{split}
\end{equation}
Here $(x,y)^T\in M\times\mathbb{R}^m$ is regarded as the state and $v\in\mathcal{U}^m$ the control of the extended system \eqref{Eq3}. We shall prove, in the next section, the equivalence between controllability of system \eqref{Eq2} and \eqref{Eq3}, that is,
\begin{Theorem}\label{Thr1}
	Nonlinear system \eqref{Eq2} is globally controllable if and only if affine system \eqref{Eq3} is globally controllable.
\end{Theorem}

We shall also prove that if $M$ and $\mathcal{C}$ are compact, then
\begin{Corollary}\label{Coro2}
	Nonlinear system \eqref{Eq2} is globally controllable with input $u$ valued in the bounded set $\in\mathcal{C}$ if and only if system \eqref{Eq3} with $(x,y)^T\in M\times\mathcal{C}$ is globally controllable with bounded input $v$.
\end{Corollary}

Here, we give examples to illustrate some applications of our main result.
\begin{Example}
	Consider the following linear system 
	\begin{equation}\label{Eq4}
	\dot{x}=Ax+Bu,\\
	\end{equation}
	where $x=(x_1,x_2,x_3)^T,A=(a_{ij}),1\le i,j\le3$ and $B=(0,0,1)^T$.
	
	According to Theorem \ref{Thr1}, system \eqref{Eq4} is controllable if and only if the following system 
	\begin{equation}\label{Eq5}
	\begin{pmatrix}
	\dot{x}_1\\
	\dot{x}_2\\
	\end{pmatrix}=\hat{A}\begin{pmatrix}
	x_1\\
	x_2\\
	\end{pmatrix}+\begin{pmatrix}
	a_{13}\\
	a_{23}\\
	\end{pmatrix}v,
	\end{equation}
	where $\hat{A}=(a_{ij}),1\le i,j\le2$, is controllable. For the above system to be controllable, $a_{13}^2+a_{23}^2$ has to be positive. Therefore, we 
	can apply the change of coordinates
	$$\begin{pmatrix}
	y_1\\
	y_2\\
	\end{pmatrix}=P\begin{pmatrix}
	x_1\\
	x_2\\
	\end{pmatrix}$$
	to \eqref{Eq5} and obtain 
	$$\begin{pmatrix}
	\dot{y}_1\\
	\dot{y}_2
	\end{pmatrix}=\bar{A}\begin{pmatrix}
	y_1\\
	y_2\\
	\end{pmatrix}+\begin{pmatrix}
	0\\
	a_{13}^2+a_{23}^2\\
	\end{pmatrix}v,$$
	where $\bar{A}=P\hat{A}P^{-1}$ and 
	$$P=\begin{pmatrix}
	a_{23}&-a_{13}\\
	a_{13}&a_{23}\\
	\end{pmatrix}.$$
	Using Theorem \ref{Thr1} again, we conclude that system \eqref{Eq4} is controllable if and only if 
	\begin{equation}\label{Eq6}
	\bar{A}_{12}\ne0.
	\end{equation}
	It can be seen that criterion \eqref{Eq6} is equivalent to the controllability matrix criterion $\mathrm{dim}(B,AB,A^2B)=3$.
\end{Example}

\begin{Example}
	Consider the following affine system 
	\begin{equation}\label{Eq7}
	\begin{split}
	&\dot{x}_1=\sin{x_3},\\
	&\dot{x}_2=\cos{x_3},\\
	&\dot{x}_3=x_4,\\
	&\vdots\\
	&\dot{x}_{n-1}=x_n,\\
	&\dot{x}_{n}=u.
	\end{split}
	\end{equation}
	It is clear that system \eqref{Eq7} is controllable (see Theorem 6.2 in \cite{Cheng2010Analysis}). Using Theorem \ref{Thr1}, this claim can be verified easily. Indeed, repeatedly use Theorem \ref{Thr1} $n-3$ times we conclude that the controllability of system \eqref{Eq7} is equivalent to that of 
	\[\begin{split}
	&\dot{x}_1=\sin{v},\\
	&\dot{x}_2=\cos{v},\\
	\end{split}\]
	which is obviously controllable. Hence, system \eqref{Eq7} itself is controllable.
\end{Example}

The previous two examples show the necessary of Theorem \ref{Thr1}, the following example which is given in \cite{Emel'yanov1992Application} will show the sufficiency of Theorem \ref{Thr1}.

\begin{Example}
	Consider the nonlinear system 
	\begin{equation}\label{Eq8}
	\begin{split}
	\dot{x}_1&=u,\\
	\dot{x}_2&=x_3^3,\\
	\dot{x}_3&=u^3,
	\end{split}
	\end{equation}
	which is not affine. We use Theorem \ref{Thr1} to verify the controllability of system \eqref{Eq8}. Extending it to the affine system
		\[\begin{split}
		\dot{x}_1&=x_4,\\
		\dot{x}_2&=x_3^3,\\
		\dot{x}_3&=x_4^3,\\
		\dot{x}_4&=v,
		\end{split}\]
		of which the controllability can be verified by Theorem 6.2 in \cite{Cheng2010Analysis},  then the controllability of system \eqref{Eq8} can be established by applying Theorem \ref{Thr1}.
\end{Example}

\section{The Proofs of Main results}\label{Sec4}

To prove Theorem \ref{Thr1}, we study the controllability of system \eqref{Eq3}. Note that the state space $M\times \mathbb{R}^m$ of it admits a nature
foliation $\mathcal{F}=\left\{M\times\left\{y\right\}|y\in\mathbb{R}^m\right\}$, and the vector field $(f(\cdot,\cdot),0)$ is tangent to the leaf $L_y\coloneqq 
M\times\left\{y\right\}$ for any $y\in\mathbb{R}^m$. Therefore, we can restrict the vector field $(f(\cdot,\cdot),0)$ to the submanifold $L_y$ of $M\times\mathbb{R}^m$ to obtain a vector field $(f(\cdot,y),0)$ on $L_y$. Moreover, let $g_i,i=1,\cdots,m$ be the vector field associated with the $i$-th input $v_i$, that is $g_i=\partial_{y_i}$, and $G=\left\{g_1,\cdots,g_m\right\}$, then clearly the mapping $\varphi:L_{y}\rightarrow L_{y_0}$ defined by 
$$\varphi(p)=e^{tg}(p),\quad\forall p\in L_y$$
where $g=(g_1,\cdots,g_m)\in G^m$ and $t=y_0-y\in\mathbb{R}^m$, is a diffeomorphism. Hence, we can push-forward the vector fields $(f(\cdot,y),0)$ on all leaves $L_y$ of $\mathcal{F}$ to a single leaf $L_{y_0}$, resulting in a family of vector fields
$$\bar{F}\coloneqq\left\{\varphi_\ast(f(\cdot,y),0)\mid y\in\mathbb{R}^m\right\}
=\left\{(f(\cdot,y),0)|y\in\mathbb{R}^m\right\}$$
on $L_{y_0}$ for each $y_0\in\mathbb{R}^m$.

Note that the family of vector fields $\bar{F}$ on $L_{y_0}$ can be identified with the family of vector fields $F\coloneqq\left\{f(\cdot,u)|u=\text{const}\right\}$ on $M$. Indeed, let $\pi$ be the projection mapping from $L_{y_0}$ to $M$, i.e.,
$$\pi:(x,y_0)\mapsto x.$$
It is clear that $\pi$  is a diffeomorphism. Moreover, the tangent mapping 
$$\pi_\ast:(f(x,y),0)\mapsto f(x,y),$$
induced by $\pi$ defines a one-to-one correspondence between $\bar{F}$ and $F$.

In addition, we need the following lemma which follows directly from the proof of Lemma 1 in \cite{Jouan2009Finite}.
\begin{Lemma}\label{Lem3}
	If system \eqref{Eq2} is controllable, then for any point $x\in M$, there exist a integer $l>0$, a subset $U$ of $\mathbb{R}_+^l$ and $X\in F^l$, such that the mapping $\phi:U\rightarrow M$ defined by 
	$$\phi(t)\coloneqq e^{tX}(x)$$
	satisfies that the image $\phi(U)$ of $U$ under $\phi$ contains a non-empty open subset of $M$.
\end{Lemma}

\begin{figure}
	\begin{center}
		\includegraphics[height=4cm]{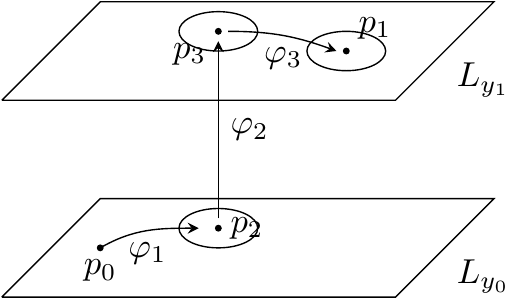}    % The printed column  
		\caption{}  % width is 8.4 cm.
		\label{Fig1}                                 % Size the figures 
	\end{center}                                 % accordingly.
\end{figure}

\begin{proof}(Proof of Theorem \ref{Thr1})
	The sufficiency of Theorem \ref{Thr1} is obvious. Indeed, if system \eqref{Eq3} is controllable, then system \eqref{Eq2} is controllable by piecewise continuous control and consequently by piecewise constant control in view of Theorem 3 in \cite{Jouan2009Finite}.
	
	The necessity is proved in two steps. 
	
	\textbf{Step 1.} We prove that for arbitrary given two points $p_0=(x_0,y_0)$ and $p_1=(x_1,y_1)\in M\times\mathbb{R}^m$, there exists a mapping $\varphi:U\times I^m\rightarrow M\times\mathbb{R}^m$, where $U$ is a subset of $\mathbb{R}_+^l$, such that the image $\varphi(U\times I^m)$ is a neighborhood of $p_1$.
	
	Indeed, since system \eqref{Eq2} is controllable, there exist $X\in F^l$ and a subset $U$ of $\mathbb{R}_+^l$, thanks to Lemma \ref{Lem3}, such that the image of $U$ under the mapping $\phi:U\rightarrow M$ defined by 
		$$\phi(t)\coloneqq e^{tX}(x_0),\quad \forall t\in U$$
		contains a non-empty open subset of $M$. Let $x_2$ be a point in this open subset, then $\phi(U)$ is a neighborhood of $x_2$. Now, we lift $\phi$ to the mapping $\varphi_1:U\rightarrow L_{y_0}$ by
	$$\varphi_1(t)\coloneqq\pi^{-1}\circ\phi(t),\quad \forall t\in U,$$
	where $\pi^{-1}$ is the inverse of the projection mapping $\pi$ defined previously. It can be verified that 
	$$\varphi_1(t)=e^{t\bar{X}}(p_0),$$ 
	where $\bar{X}=(\pi_\ast^{-1}X_1,\cdots,\pi_\ast^{-1}X_l)\in\bar{F}^l$. Clearly, $\varphi_1(U)$ is a neighborhood (with respect to the subspace topology of $L_{y_0}$) of $p_2\coloneqq \pi^{-1}(x_2)=(x_2,y_0)$ (see Fig. 1).
	
	Let $\eta=y_1-y_0\in\mathbb{R}^m$ then the mapping $\varphi_2$ defined by 
		$$\varphi_2(p)\coloneqq e^{\eta g}(p),\quad\forall p\in L_{y_0}.$$
		is a diffeomorphism between $L_{y_0}$ and $L_{y_1}$. 
	
	Next, let $p_3\coloneqq\varphi_2(p_2)$, then $p_3\in L_{y_1}$ and hence can be written as $p_3=(x_3,y_1)$ where $x_3\in M$ and $y_1\in\mathbb{R}^m$. By the definition of controllability of system \eqref{Eq2}, there exists a feasible control $u:[0,T]\rightarrow\mathbb{R}^m$ such that the associated trajectory $x(t)$ satisfies $x(0)=x_3$ and $x(T)=x_1$. Since $u$ is piecewise constant, there exists a partition $0=T_0<T_1<\cdots<T_k=T$ of $[0,T]$ such that $u$ is constant on each interval $[T_{i-1},T_i),i=1,\cdots,k$. Let $u(i)=u(T_{i-1}), Y_i=f(x,u(i))\in F$ and $\tau_i=T_i-T_{i-1}$ for $i=1,\cdots,k$, then we have 
	$$x_1=e^{\tau Y}(x_3),$$
	where $Y=(Y_k,\cdots,Y_1)\in F^k$ and $\tau=(\tau_k,\cdots,\tau_1)\in\mathbb{R}_+^k$. Analogously, we can lift $e^{\tau Y}:M\rightarrow M$ to the mapping 
	$\varphi_3:L_{y_1}\rightarrow L_{y_1}$ by
	$$\varphi_3(p)\coloneqq e^{\tau\bar{Y}}(p),\quad\forall p\in L_{y_1},$$
	where $\bar{Y}=(\tilde{\pi}_\ast^{-1}Y_k,\cdots,\tilde{\pi}_\ast^{-1}Y_1)\in\bar{F}^k$and $\tilde{\pi}$ is the projection mapping from $L_{y_1}$ to $M$. 
	Clearly, the mapping $\varphi_3$ is a diffeomorphism and satisfies $\varphi_3(p_3)=p_1$.
	
	Since $\varphi_2$ and $\varphi_3$ are diffeomorphisms, the image $\varphi_3\circ\varphi_2\circ\varphi_1(U)$ of $U$ is a neighborhood of $p_1$ with respect to the subspace topology of $L_{y_1}$. Hence, if we define the mapping $\varphi:U\times I^m\rightarrow M\times \mathbb{R}^m$ by
	$$\varphi(t,s)\coloneqq e^{sg}\circ\varphi_3\circ\varphi_2\circ\varphi_1(t), \quad\forall t\in U, s\in I^m,$$
	then the image $\varphi(U\times I^m)$ of $U\times I^m$ under $\varphi$ is a neighborhood of $p_1$ with respect to the topology of $M\times\mathbb{R}^m$. 	 
	
	\textbf{Step 2.} We prove in this step that the mapping $\varphi$ can be modified in a manner so that it is realizable by system \eqref{Eq3} and the image $\varphi(U\times I^m)$ remains a neighborhood of $p_1$, thereby proving that $p_1$ is reachable by $p_0$ and consequently the controllability of system \eqref{Eq3} in view of the arbitrariness of $p_0$ and $p_1$. To this end, note that the vector fields $V$ appear in $\varphi$ falling into the following two categories.
	
	Case 1. $V\in G$, that is $V=g_i$ for some $i\in\left\{1,\cdots,m\right\}$. In this case, we have 
	$$e^{\sigma V}=e^{\sigma/{v_i}(((f,0)+g_iv_i)-(f,0))}.$$
	Choosing $v_i$ so that $\sigma/v_i\ge0$ and defining  
	$$\rho(\sigma)\coloneqq e^{\sigma/{v_i}((f,0)+g_iv_i)},$$
	then $\rho(\sigma)$ is physically realizable by system \eqref{Eq3}. Replacing $e^{\sigma V}$ in $\varphi$ by $\rho(\sigma)$ as long as $V\in G$, we will obtain another mapping $\varphi^\prime: U\times I^m\rightarrow M\times \mathbb{R}^m$. In general, the image $\varphi^\prime(U\times I^m)$ of $U\times I^m$ under $\varphi^\prime$ will be different from that under $\varphi$. But as long as $|v_i|$ is large enough, the image $\varphi^\prime(U\times I^m)$ will remain a neighborhood of $p_1$. 
	
	Case 2. $V\in \bar{F}$, that is $V=e^{\beta h}_\ast(f,0)$, where $\beta=(\beta_1,\cdots,\beta_j)\in\mathbb{R}^j$ and $h=(h_1,\cdots,h_j)\in G^j$. In this case, we have (see Lemma 2.20 in \cite{Agrachev2012Introduction})
	$$e^{\sigma V}=e^{\beta h}\circ e^{\sigma(f,0)}\circ e^{-\hat{\beta}\hat{h}},$$
	where $\hat{\beta}=(\beta_j,\cdots,\beta_1)$ and $\hat{h}=(h_j,\cdots,h_1)$. Note that $e^{\sigma(f,0)}$ is physically realizable since $\sigma$ is positive, thus no modification is required. The proof is complete once we notice that the vector fields $\pm h_i,i=1,\cdots,j$ falling into Case 1.
\end{proof} 

\begin{proof}[Proof of Corollary \ref{Coro2}]
	The ``if'' part is obvious. To prove the ``only if'' part, we suppose that system \eqref{Eq2} is globally controllable with input $u$ valued in the bounded set $\mathcal{C}$ and let $p_0\in M\times\mathcal{C}$ be fixed. By the proof of Theorem \ref{Thr1}, for any $p\in M\times\mathcal{C}$, $p_0$ can reach all the points in an open neighborhood $W_p$ of $p$ with input $v$ valued in a bounded set $\mathcal{B}_p$. Since $\{W_p\}_{p\in M\times\mathcal{C}}$ is an open cover of the compact set $M\times\mathcal{C}$, there exist a finte number of points $p_1,\cdots,p_k$ such that $\bigcup_{i=1}^kW_{p_i}$ covers $M\times\mathcal{C}$ and consequently, $p_0$ can reach any point in $M\times\mathcal{C}$ with input $v$ valued in the bounded set $\bigcup_{i=1}^k\mathcal{B}_{p_i}$. Applying the same argument to the time-reversed systems of \eqref{Eq2} and \eqref{Eq3}, one can see that $p_0$ can also be reached by any point in $M\times\mathcal{C}$ with bounded input. Put these two facts together, we conclude that system \eqref{Eq3} is globally controllable with bounded input $v$.
\end{proof}

\begin{Remark}
	Theorem \ref{Thr1} was proved in the situation where $\mathcal{U}$ is the set of piecewise constant functions, but it can be extended to other situations where $\mathcal{U}$ is more general. Since controllability using arbitrary controls is equivalent to controllability using just piecewise constant controls (see \cite{Jouan2009Finite,Sussmann1979Subanalytic,Sontag1988Finite}for details).  
\end{Remark}

\begin{Remark}
	Compared with the approach of \cite{Cheng2010Analysis} in their proofs of Theorem 6.1 and 6.2, the proposed one is novel in the following way. In the approach of \cite{Cheng2010Analysis}, the state space is foliated by submanifolds which are tangent to $G$ 	(i.e. maximal integral submanifolds of $G$), therefore the ``controllability'' within the leaves is guaranteed while that traverses the leaves is not. Hence, the main task is to prove the flows can traverse the leaves in all directions. But in our approach, the state space is foliated by submanifolds perpendicular to $G$, therefore the motions traverse the leaves can be in all directions and the problem is to prove the ``controllability'' within the leaves.
\end{Remark}

\section{Conclusion}\label{Sec5}
In this note, we have proved that the global controllability of a general nonlinear system is preserved when an integrator is added. Also, we have extended this result to the case with bounded inputs. By making use of our result, the controllability problem of a general nonlinear system can be converted to that of an affine system which has been studied more thoroughly. The power of our result has been demonstrated by several examples. In addition, the differences between our approach and existing ones have been illustrated.

\section*{References}

\bibliography{mybibfile}

\end{document}